\documentclass[11pt,sumlimits,intlimits]{amsart}
\usepackage{enumerate,mathtools}
\usepackage[backref]{hyperref}
\usepackage{manfnt,amssymb}

\theoremstyle{plain}
\newtheorem{theorem}{Theorem}[section]

\newtheorem{lemma}[theorem]{Lemma}
\newtheorem{proposition}[theorem]{Proposition}

\theoremstyle{definition}
\newtheorem{definition}[theorem]{Definition}
\newtheorem{remark}[theorem]{Remark}

\numberwithin{equation}{section}

\DeclareMathOperator{\Ad}{Ad}

\begin{document}

\title{Approximately Multiplicative Decompositions of Nuclear Maps}

\thanks{}

\author[D.\ Wagner]{Douglas A.\ Wagner}

\address{\hskip-\parindent Douglas A.\ Wagner. Department of
  Mathematics, Texas Christian University, Fort Worth, TX, 76129, USA.}
\email{Douglas.Wagner@TCU.edu}


\begin{abstract}
    We expand upon work from many hands on the decomposition of nuclear maps. Such maps can be characterized by their ability to be approximately written as the composition of maps to and from matrices. Under certain conditions (such as quasidiagonality), we can find a decomposition whose maps behave nicely, by preserving multiplication up to an arbitrary degree of accuracy and being constructed from order zero maps (as in the definition of nuclear dimension). We investigate these conditions and relate them to a W*-analog. 
\end{abstract} 

\date{\today}

\maketitle

\section{Background}


The germinal idea of this investigation was the definition of \emph{nuclear dimension} in \cite[Def. 2.1.]{WinterZacharias}, which quantifies the \textit{completely positive approximation property} (CPAP) of nuclear maps. (Although this notion goes back even further to \cite{KirchbergWinter}'s \emph{decomposition rank}, and so on.)

\begin{definition}\cite[Ch. 2]{BrownOzawa} \label{df:nuclear}
	A contractive completely positive (c.c.p.) map $\pi:A\rightarrow B$ between C*-algebras (or to a von Neumann algebra $B$) is called \textbf{nuclear} (resp. \textbf{weakly nuclear}) if it admits a certain decomposition: there exists a net $(F_n)$ of finite-dimensional C*-algebras and c.c.p. maps  $A\stackrel{\psi_n}{\longrightarrow}F_n\stackrel{\varphi_n}{\longrightarrow}B$
	such that $(\varphi_n\circ\psi_n)$ converges to $\pi$ in the point-norm (resp. point-$\sigma$-weak) topology. It is this decomposition to which the ``CPA'' in CPAP refers. 
	If $\pi$ is the identity, then we say that $A$ itself is \textbf{nuclear}; if $\pi$ is an inclusion, we say $A$ is \textbf{exact}.
\end{definition}

Methods for strengthening the CPAP were first explored in \cite{HirshbergKirchbergWhite}. This was synthesized with \cite{BlackadarKirchberg}'s results on quasidiagonal nuclear C*-algebras into \cite{BrownCarrionWhite}, which was built upon by \cite{CarrionSchafhauser}. This paper builds further by providing a partial answer to the final question in \cite{CarrionSchafhauser}.

\cite[Thm. 1.4]{HirshbergKirchbergWhite} examined the case of nuclear $A$, showing that we may then find $\varphi_n:F_n\rightarrow A$ (as in Definition~\ref{df:nuclear}) that are convex combinations of finitely many c.c.p. order zero maps. \cite[Thm. 3.1]{BrownCarrionWhite} then showed the net $(\psi_n:A\rightarrow F_n)$ may be chosen to be approximately order zero. Furthermore, drawing from \cite{BlackadarKirchberg}, if both $A$ and all of its traces are quasidiagonal, then $(\psi_n)$ may be chosen to be approximately multiplicative. In fact, \cite[Thm. 2.2]{BrownCarrionWhite} shows that the converse is true as well.

Using the result from \cite{HirshbergKirchbergWhite} as their starting point, \cite{CarrionSchafhauser} began the process of generalizing to non-nuclear $A$: so long as the nuclear map $\pi$ is order zero, we may find $(\psi_n)$ that is approximately order zero and $\varphi_n$ that are convex combinations of c.c.p. order zero maps \cite[Thm. 1]{CarrionSchafhauser}.
Additionally, approximate multiplicativity of $(\psi_n)$ is attainable for weakly nuclear $\pi$ with quasidiagonal $A$ \cite[Prop. 3]{CarrionSchafhauser} .

The question is, then, ``What is necessary or sufficient to get approximate multiplicativity for nuclear $\pi$?'' 
This paper's results are two necessary conditions.
\begin{lemma}\label{lm:trace}
	Let $\pi:A\rightarrow B$ be a \textasteriskcentered-homomorphism between C*-algebras $A,B$ that admits an approximately multiplicative norm-decomposition (see Definition~\ref{df:AMND}), and $\tau$ be a trace on $\pi(A)$. Then $\tau\circ\pi$ is a quasidiagonal trace on $A$.
\end{lemma}
\begin{theorem}\label{th:nonexact}
	Let $\pi:A\rightarrow B\subseteq B^{**}$ be a \textasteriskcentered-homomorphism that admits an approximately multiplicative $\sigma$-strong*-decomposition (see Definition~\ref{df:AMSSD}). 
	Then the inclusion $\pi(A)\subseteq B$ is nuclear.
\end{theorem}
We also reach a full characterization in the exact case, which stems from a more general sufficient condition given in Proposition~\ref{pr:locallysplit}.
\begin{theorem}\label{th:exact}
	Let $A$ be an exact C*-algebra and $\pi:A\rightarrow B$ be a \textasteriskcentered-homomorphism to another C*-algebra $B$. Then $\pi$ admits an approximately multiplicative norm-decomposition iff it is nuclear and quasidiagonal and $\tau\circ\pi$ is a quasidiagonal trace on $A$ for every trace $\tau$ on $\pi(A)$.
\end{theorem}


We shall make use that the $\sigma$-strong* topology on a von Neumann algebra $N$ agrees on bounded subsets with the topology generated by seminorms
$$\|x\|_{\rho}:= \sqrt{\rho\left(\frac{x^*x+xx^*}{2}\right)}$$
for a separating family of normal states $\rho$ on $N$ (see e.g. \cite[III.2.2.19]{Blackadar06}).

For simplicity, we shall assume that our C*-algebras are separable and unital. 

\section{Results}\label{sc:epsilon}

Our results are only possible through use of the following as-yet unpublished result, which is something of a folklore theorem (cf. \cite{CETW}). 
We thank the authors of \cite{CCESTW} for allowing us to reproduce the proof here. 

\begin{theorem}[\cite{CCESTW}]\label{th:CCESTW} 
	Let $\theta,\pi:A\rightarrow N$ be weakly nuclear \textasteriskcentered-homomorphisms from a C*-algebra $A$ to a finite von Neumann algebra $N$ that agree on traces (that is, $\tau\circ\theta= \tau\circ\pi$ for every trace $\tau$ on $N$). Then $\theta$ and $\pi$ are strong* approximately unitarily equivalent.
\end{theorem}
\begin{proof}
	Let normal trace $\tau_0$ on $N$ be given. 
	Replacing $N$ with $\pi_{\tau_0}(N)$ (where $\pi_\tau$ is the GNS representation corresponding to $\tau_0$) if necessary, we may assume $\tau_0$ is faithful. 
	We need to show that $\theta$ and $\pi$ are unitarily equivalent as maps into the tracial ultrapower $N^\omega_\tau$ of $N$ with respect to $\tau_0$.
	
	Define the weakly nuclear \textasteriskcentered-homomorphism $\mu:A\rightarrow\mathbb{M}_2(N^\omega_\tau)$ by 
	$$\mu(a)= \begin{bmatrix}\theta(a)&0\\0&\pi(a)\end{bmatrix}.$$
	We claim that $M:=\pi(A)''$ is hyperfinite.
	
	Indeed, let $\psi_n:A\rightarrow F_n$ and $\varphi_n:F_n\rightarrow M$ be c.c.p. maps for finite-dimensional C*-algebras $F_n$ such that $(\varphi_n\circ\psi_n)$ converges in the point-weak* topology to $\mu$. Fix a unital normal representation $M\subseteq\mathbb{B}(\mathcal{H})$, and define maps $\eta,\eta_n:A\otimes_{\mathrm{alg}} M'\rightarrow \mathbb{B}(\mathcal{H})$ by $\eta(a\otimes b)= \mu(a)b$ and ${\eta_n(a\otimes b)}= {\varphi_n\circ\psi_n(a)b}$, so that $(\eta_n)$ converges to $\eta$ in the point-weak* topology. 
	The $\eta_n$ are continuous with respect to the minimal tensor product since they factor through $F_k\otimes M'$, hence so is $\eta$. A conditional expectation of  $\mathbb{B}(\mathcal{H})$ onto $M'$ is then provided by \cite[Prop. 3.6.5]{BrownOzawa}, confirming that $M$ is injective, hence hyperfinite by Connes' Theorem \cite[Thm. 6]{Connes76}.
	
	Define projections 
	$$p_1= \begin{bmatrix}1&0\\0&0\end{bmatrix},\hspace{1cm}p_2= \begin{bmatrix}0&0\\0&1\end{bmatrix}$$ in $\mathbb{M}_2(N^\omega_\tau)\cap\pi(A)'$. Then $\tau(p_1x)= \tau(p_2x)$ for every normal trace $\tau$ and every $x\in\pi(A)$, hence also for every $x\in\pi(A)''$. By \cite[Lem. 4.5]{CETWW}, $\tau(p_1)= \tau(p_2)$ for every trace $\tau$ on $\mathbb{M}_2(N^\omega_\tau)\cap\pi(A)'$. Thus there is a unitary $u=[u_{i,j}]\in\mathbb{M}_2(N^\omega_\tau)\cap\pi(A)'$ such that $u^*p_1u=p_2$, hence $u_{1,1}= 0$ and $u_{1,2}^*u_{1,2}= 1_{N^\omega_\tau}$. 
	Therefore $u_{1,2}$ is unitary since $N^\omega_\tau$ is finite, and $\theta(a)u_{1,2}=u_{1,2}\pi(a)$ since $u\in\pi(A)'$.
\end{proof}

Most of the following lemma's proof is borrowed from \cite{BrownCarrionWhite}; improvements are due to mixing in material from \cite{CarrionSchafhauser} and Theorem~\ref{th:CCESTW}. To properly state the theorems, we need (the W* half of) the paper's central definition.

\begin{definition}\label{df:AMSSD}
	Let $\pi:A\rightarrow N$ be a \textasteriskcentered-homomorphism from a C*-algebra $A$ to a von Neumann algebra $N$. An \textbf{approximately multiplicative $\sigma$-strong*-decomposition} of $\pi$ is a net of u.c.p. maps $A\stackrel{\psi_n}{\longrightarrow}F_n\stackrel{\varphi_n}{\longrightarrow}N$ for finite-dimensional C*-algebras $F_n$ such that
	\begin{itemize}
		\item[(i)] $\varphi_n\circ\psi_n(x)\rightarrow \pi(x)$ in the $\sigma$-strong* topology for all $x\in A$,
		\item[(ii)] $\left\|\psi_n(x)\psi_n(y)-\psi_n(xy)\right\|\rightarrow 0$ for all $x,y\in A$,
		\item[(iii)] every $\varphi_n$ is a \textasteriskcentered-homomorphism.
	\end{itemize}
\end{definition}

\begin{lemma}\label{lm:vonNeumann}
	Let $A$ be a quasidiagonal C*-algebra and $\pi:A\rightarrow N$ be a weakly nuclear \textasteriskcentered-homomorphism such that, for every trace $\tau$ on $\pi(A)$, the trace $\tau\circ\pi$ is quasidiagonal. Then $\pi$ admits an approximately multiplicative $\sigma$-strong*-decomposition.
\end{lemma}
\begin{remark}
	Note that the ``quasidiagonal $\tau\circ\pi$'' condition is satisfied if either every trace on $A$ or every trace on $\pi(A)$ is quasidiagonal. As Lemma~\ref{lm:trace} alludes, this property is essential. 
\end{remark}
\begin{proof}[Proof of Lemma~\ref{lm:vonNeumann}]
		Let $\mathcal{F}\subset A$ be a finite set of contractions, $\mathcal{S}$ a finite set of normal states on $N$, and $\epsilon>0$. 
	Define normal state $\rho= \sum_{\rho'\in\mathcal{S}}\rho'/|\mathcal{S}|$, so that $\|b\|_{\rho'}^2\leq |\mathcal{S}|\|b\|_\rho^2$ for every $\rho'\in\mathcal{S}$ and $b\in N$. Thus we need only make reference to $\rho$ in the proof, rather than the entire set $\mathcal{S}$. 
	
	$N$ may be decomposed into $N_{1}\oplus N_\infty$ for von Neumann algebras $N_{1},N_\infty$ that are respectively finite and properly infinite. Similarly, $\pi= \pi_{1}\oplus \pi_\infty$ for weakly nuclear \textasteriskcentered-homomorphisms $\pi_{1}:A\rightarrow N_{1}$ and $\pi_{\infty}:A\rightarrow N_{\infty}$. 
	We shall deal with each summand separately by assuming it composes all of $N$.

	\textbf{Suppose $N$ is properly infinite.} We begin by following along with the proof of \cite[Prop. 3]{CarrionSchafhauser}. 
	By weak nuclearity, there are $k\in\mathbb{Z}^+$ and u.c.p. maps $A\stackrel{\psi'}{\longrightarrow}\mathbb{M}_{k}\stackrel{\dot{\varphi}'}{\longrightarrow}N$ such that 
	$\|\dot{\varphi}'\circ\psi'(x)-\pi(x)\|_{\rho}\leq\epsilon$ for every $x\in\mathcal{F}$. Since $A$ is separable, we may fix a faithful unital representation $A\subseteq\mathbb{B}(\mathcal{H})$ such that $\mathcal{H}$ is separable and $A$ contains no nonzero compact operators. Voiculescu's Theorem \cite[Thm. II.5.3]{Davidson96} provides an isometry $v:\mathbb{C}^{k}\rightarrow\mathcal{H}$ such that $\|v^*xv-\psi'(x)\|\leq\epsilon$ for every $x\in\mathcal{F}$. Likewise, quasidiagonality provides a finite-rank projection $p\in \mathbb{B}(\mathcal{H})$ such that $\|pv-p\|\leq\epsilon$ and $\|px-xp\|\leq\epsilon$ for every $x\in\mathcal{F}$. Define u.c.p. maps $A\stackrel{\psi}{\longrightarrow}\mathbb{B}(p\mathcal{H})\stackrel{\dot{\varphi}}{\longrightarrow}N$ by $\psi(a)= pap$ and $\dot{\varphi}(T)= \dot{\varphi}'(v^*Tv)$. Thus, for all $x,y\in\mathcal{F}$,
	\begin{align*}
	\|\psi(x)\psi(y)-\psi(xy)\|&\leq \|p\|\|xp-px\|\|yp\|\leq\epsilon,\\
	\|\dot{\varphi}\circ\psi(x)-\pi(x)\|_{\rho}&
	\leq \|\dot{\varphi}'\|\|v^*\psi(x)v-\psi'(x)\|+\epsilon
	\\&\leq \|v^*p-v^*\|\|xpv\|+\|v^*x\|\|pv-v\|+2\epsilon\leq 4\epsilon.
	\end{align*}
	
	We now follow the proof of \cite[Lem. 2.4]{BrownCarrionWhite}, beginning on page 50 with the explanation that the reference itself is following the proof of \cite[Prop. 2.2]{Haagerup85}. Our properly infinite assumption finally kicks in, allowing us to find a unital embedding $\iota:\mathbb{B}(p\mathcal{H})\rightarrow N$
	. By \cite[Prop. 2.1]{Haagerup85} there is isometry $w\in N$ such that $\dot{\varphi}(T)= w^*\iota(T)w$ for all $T\in \mathbb{B}(p\mathcal{H})$, and \cite[Page 167]{Haagerup85} shows how $w$ may be approximated by a unitary $u$ so that 
	$$\|\Ad(u^*)\circ\iota(T)-\dot{\varphi}(T)\|_\rho\leq \|T\|\epsilon\leq \epsilon$$
	for every $T\in \psi(\mathcal{F})\subset \mathbb{B}(p\mathcal{H})$. The \textasteriskcentered-homomorphism $\varphi= \Ad(u^*)\circ\iota$ therefore completes this term of the net.
	
	\textbf{Now suppose $N$ is finite.} This proof follows exactly in the footsteps of that of \cite[Lem. 2.5]{BrownCarrionWhite}, up until its last two paragraphs. At that point we note that $\theta$ is nuclear by definition: we have finite-dimensional u.c.p. maps  $\widetilde{\psi}_n:=\bigoplus_{k=1}^n\psi_k:A\rightarrow\bigoplus_{k=1}^nF_k$ and \textasteriskcentered-homomorphisms  $\widetilde{\phi}_n:\bigoplus_{k=1}^nF_k\rightarrow N^\omega$ given by 
	\begin{align*}
	\widetilde{\phi}_n\big(\left(T_1,\ldots,T_n\right)\big)= q\big(\left(\phi_1(T_1),\ldots,
	\phi_{n-1}(T_{n-1}),\phi_{n}(T_{n}),\phi_{n}(T_n),\ldots \right)_n\big).
	\end{align*}
	Thus $\theta= \lim_\omega\phi_n\circ\psi_n= \lim_n\widetilde{\phi}_n\circ\widetilde{\psi}_n$. This allows us to use Theorem~\ref{th:CCESTW} to deduce that $\theta$ is strong* approximately unitarily equivalent to $\pi^\omega$, and the proof is finished using Kirchberg's $\epsilon$-test as in the original.	
\end{proof} 

We introduced this form of W*-decomposition solely to relate it to the following C*-counterpart:

\begin{definition}\label{df:AMND}
	Let $\pi:A\rightarrow B$ be a \textasteriskcentered-homomorphism between C*-algebras $A,B$. An \textbf{approximately multiplicative norm-decomposition} of $\pi$ is a sequence of u.c.p. maps $A\stackrel{\psi_n}{\longrightarrow}F_n\stackrel{\varphi_n}{\longrightarrow}B$ for finite-dimensional C*-algebras $F_n$ such that
	\begin{itemize}
		\item[(i)] $\left\|\varphi_n\circ\psi_n(x)-\pi(x)\right\|\rightarrow 0$ for all $x\in A$,
		\item[(ii)] $\left\|\psi_n(x)\psi_n(y)-\psi_n(xy)\right\|\rightarrow 0$ for all $x,y\in A$,
		\item[(iii)] every $\varphi_n$ is a convex combination of u.c.p. order zero maps.
	\end{itemize}
\end{definition}

\begin{proof}[Proof of Lemma~\ref{lm:trace}]
	The composition of a trace with an order zero map is again (rescalable into) a trace \cite[Cor. 3.4]{WinterZacharias}, so each $\tau\circ\varphi_n$ is a trace on the finite-dimensional C*-algebra $F_n$. By assumption, the $\psi_n$ are approximately multiplicative, and $\tau\circ\varphi_n\circ\psi_n\rightarrow \tau\circ\pi$ in the weak* topology. Therefore $\tau\circ\pi$ is quasidiagonal.
\end{proof}

\begin{proof}[Proof of Theorem~\ref{th:nonexact}]
	Let $A\stackrel{\psi_n}{\longrightarrow}F_n\stackrel{\varphi_n}{\longrightarrow}B^{**}$ be an approximately multiplicative $\sigma$-strong*-decomposition. 
	By Arveson's Extension Theorem we have conditional expectations $E_n:B^{**}\rightarrow\varphi_n(F_n)$. Moreover, by Alaoglu's Theorem we may pass $(E_n)$ to a subsequence that converges to a linear map $E:B^{**}\rightarrow B^{**}$ in the point-$\sigma$-weak topology (cf. \cite[Thm. 1.3.7]{BrownOzawa}).
	
	Let $\epsilon>0$, $a\in A$, and normal functional $\eta\in B^*$ all be given. 	
	Then $\eta\circ E$ is also a normal functional, so by the previous paragraph there exists $m\in\mathbb{N}$ such that $n>m$ implies each of the following hold:
	\begin{align*}
	\big|\eta\big(E_n\big(\pi(a)\big)-E\big(\pi(a)\big)\big)\big|&\leq\epsilon/4,\\
	\big|\eta\circ E\big(\pi(a)-\varphi_n\circ\psi_n(a)\big)\big|&\leq\epsilon/4,\\
	\big|\eta\big(E\big(\varphi_n\circ\psi_n(a)\big)-E_n\big(\varphi_n\circ\psi_n(a)\big)\big)\big|&\leq\epsilon/4,\\
	\big|\eta\big(\varphi_n\circ\psi_n(a)-\pi(a)\big)\big|&\leq\epsilon/4.
	\end{align*}
	Combined, we get
	\begin{align*}
	\big|\eta\big(E_n\circ\pi(a)-\pi(a)\big)\big|&\leq\epsilon,
	\end{align*}
	therefore the inclusion $\pi(A)\subseteq B^{**}$ is weakly nuclear. 
	
	We now use a Hahn-Banach argument akin to the proof (as in \cite[Prop. 2.3.6]{BrownOzawa}) that semidiscrete $A^{**}$ implies nuclear $A$. Let $X$ be the set of all maps from $\pi(A)$ to $B^{**}$ of the form $\varphi\circ\psi$ for u.c.p. maps $\psi:\pi(A)\rightarrow F$ and $\varphi:F\rightarrow B^{**}$ and finite-dimensional $F$. Then $X$ is convex.
	
	Indeed, let $w\in(0,1)$ and $\varphi_0\circ\psi_0,\varphi_1\circ\psi_{1}\in X$. Then $\psi:=\psi_0\oplus\psi_1:\pi(A)\rightarrow F_0\oplus F_1$ is u.c.p., as is the map $\varphi:F_0\oplus F_1\rightarrow B^{**}$ given by $\varphi\big((M_0,M_1)\big)= w\varphi_0(M_0)+(1-w)\varphi_1(M_1)$. Thus $w\varphi_0\circ\psi_0+(1-w)\varphi_1\circ\psi_1= \varphi\circ\psi\in X$.
	
	Let $\mathcal{F}=\{b_i~|~i\in[1,k]\}\subset \pi(A)$ be given. By weak nuclearity, the $k$-tuple $(b_i)_{i=1}^k\in \bigoplus_{i=1}^kB^{**}$ is in the weak-closure of the set $\left\{\left.\big(\varphi\circ\psi(b_i)\big)_{i=1}^k~\right|~\varphi\circ\psi\in X\right\}$, hence also in its norm-closure by the Hahn-Banach Theorem. Therefore there is a map $\varphi\circ\psi\in X$ such that $$\max_i\left\|b_i-\varphi\circ\psi(b_i)\right\|= \left\|(b_i)_{i=1}^k-\big(\varphi\circ\psi(b_i)\big)_{i=1}^k\right\|<\epsilon.$$
\end{proof}

\begin{proposition}\label{pr:different}
	Let $\pi:A\rightarrow B$ be a \textasteriskcentered-homomorphism between C*-algebras $A,B$.
	Then  $\pi$ admits an approximately multiplicative norm-decomposition iff $\pi:A\rightarrow B^{**}$ admits an approximately multiplicative $\sigma$-strong*-decomposition.
\end{proposition}
\begin{proof}
	$\mathbf{\Rightarrow}$: We retread the proof of Lemma~\ref{lm:vonNeumann}, beginning by letting normal state $\rho$, finite $\mathcal{F}\subset A$, and $\epsilon>0$ be given.

	We may skip the first paragraph of the properly infinite case, instead using the assumed approximately multiplicative norm-decomposition to provide $A\stackrel{\psi}{\longrightarrow}F\stackrel{\dot{\varphi}}{\longrightarrow}B$ that satisfy the necessary inequalities. We now need only use \cite{Haagerup85} to find a unitary $u\in N$ so that $\varphi= \Ad(u^*)\circ\iota$ approximates $\dot{\varphi}$ for a unital embedding $\iota:F\rightarrow B^{**}$. 
	
	The finite case does not require that $A$ be quasidiagonal. Of course, property (i) of our norm-decomposition witnesses the nuclearity of $\pi$, hence $\pi:A\rightarrow B^{**}$ is weakly nuclear. Therefore Lemma~\ref{lm:trace} allows us finish this case, and this direction.
	
	$\mathbf{\Leftarrow}$: This is a perturbation of \cite[Thm. 1.4]{HirshbergKirchbergWhite}. Let $A\stackrel{\psi_n}{\longrightarrow}F_n\stackrel{\dot{\varphi}_n}{\longrightarrow}B^{**}$ be an approximately multiplicative $\sigma$-strong*-decomposition.	Using \cite[Lem. 1.1]{HirshbergKirchbergWhite}, we can find order zero u.c.p. maps $\varphi_{n}:F_n\rightarrow B$ (note the range) such that $\varphi_n\circ\psi_n(x)\rightarrow \pi(x)$ in the $\sigma$-weak topology for every $x\in A$.  We once again conclude with a Hahn-Banach argument.  
%
\end{proof}

With the equivalence of these decompositions, we hope to find a way to characterize their existence. An important component seems to be that $\pi$ is a quasidiagonal \textasteriskcentered-homomorphism.

\begin{definition}\label{df:qdsh}
	A \textasteriskcentered-homomorphism $\pi:A\rightarrow B$ is \textbf{quasidiagonal} if it factors through a quasidiagonal C*-algebra $D$. That is, there exist \textasteriskcentered-homomorphisms $\pi_1:A\rightarrow D$ and $\pi_2:D\rightarrow B$ such that $\pi_1$ is surjective and $\pi= \pi_2\circ\pi_1$.
\end{definition}

\begin{definition}\label{df:lse}
	An extension 
	$$0\longrightarrow\ker\pi\longrightarrow A\longrightarrow B\longrightarrow 0$$
	is called \textbf{locally split} if every finite subset $\mathcal{G}\subset B$ of contractions 
	admits a u.c.p. \textbf{local lifting} $\lambda:\mathrm{span}(\mathcal{G})\rightarrow A$ such that $\pi\circ\lambda(b)= b$ for every $b\in \mathrm{span}\,\mathcal{G}$.
\end{definition}

For exact $A$, every such extension is locally split (see eg. \cite[Prop. 9.1.4]{BrownOzawa}). However, it is of note that the full power of exactness is not required for the following result:

\begin{proposition}\label{pr:locallysplit}
	Let $A$ be a C*-algebra, $\pi:A\rightarrow B$ be a nuclear, quasidiagonal \textasteriskcentered-homomorphism to another C*-algebra $B$, and $D,\pi_1,\pi_2$ be as in Definition~\ref{df:lse}. Further suppose that the trace $\tau\circ\pi$ is quasidiagonal for every trace $\tau$ on $\pi(A)$, and that the extension
	$$0\longrightarrow\ker\pi_1\longrightarrow A\longrightarrow D\longrightarrow 0$$
	is locally split. Then $\pi$ admits an approximately multiplicative decomposition.
\end{proposition}
\begin{proof}
	Let $\epsilon>0$ and finite subset $\mathcal{G}\subset D$ of contractions 
	be given. Let $\lambda:\mathrm{span}(\mathcal{G})\rightarrow A$ be a local lifting. 
	By nuclearity, there are $i\in\mathbb{N}$ and u.c.p. maps $A\stackrel{\psi}{\longrightarrow}\mathbb{M}_i\stackrel{\dot{\varphi}}{\longrightarrow}B$ such that $\big\|\dot{\varphi}\circ\psi(a)-\pi(a)\big\|<\epsilon$ for every $a\in \lambda(\mathcal{G})$. Arveson's Extension Theorem then provides u.c.p. $\psi':D\rightarrow \mathbb{M}_i$ that agrees with $\psi\circ\lambda$ on $\mathrm{span}(\mathcal{G})$. Thus, for any $d\in \mathcal{G}$,
	$$\big\|\dot{\varphi}\circ\psi'(d)-\pi_2(d)\big\|= \big\|\dot{\varphi}\circ\psi\big(\lambda(d)\big)-\pi\big(\lambda(d)\big)\big\|<\epsilon.$$
	
	From this we conclude that $\pi_2$ is nuclear. This allows us to apply the properly infinite case of Lemma~\ref{lm:vonNeumann} to get approximately multiplicative decompositions of $\pi_2:D\rightarrow B^{**}$. This case may then be extended to work for $A$ itself by replacing the resultant $\psi_n$ with $\psi_n\circ\pi_1$. The finite case goes through for $A$ unmodified, resulting in approximately multiplicative decompositions of $\pi:A\rightarrow B^{**}$.
\end{proof}

However, exactness does provide us with the converse statement.

\begin{proof}[Proof of Theorem~\ref{th:exact}]
	As mentioned, Proposition~\ref{pr:locallysplit} already provides the backward direction, so we need now only address the forward one. Let 
	$A\stackrel{\psi_n}{\longrightarrow}F_n\stackrel{\varphi_n}{\longrightarrow}B$ be the approximately multiplicative norm-decomposition. 
	The decomposition itself witnesses the nuclearity of $\pi$. The quasidiagonality of every $\tau\circ\pi$ was shown in Lemma~\ref{lm:trace}. The quasidiagonality of $\pi$ itself is a consequence of \cite[Thm. 4.8]{Dadarlat97}; for the convenience of the reader, we present a distillation of the proof.
	
	Since $A$ is exact, we may treat it as a C*-subalgebra of some C*-algebra $C$ such that the inclusion is nuclear. Using Arveson's Extension Theorem, we may treat the $\psi_n$'s domains as $C$. They induce a u.c.p. map $\Psi$ from $C$ to an ultraproduct $\prod_\omega F_n$, where $(\psi_n(a))_n$ is a representative sequence of $\Psi(a)$. Note that the approximate multiplicativity of $(\psi_n)$ makes $\Psi|_A$ a \textasteriskcentered-homomorphism. Likewise, we have a c.c.p. map $\Phi$ from $\prod_\omega F_n$ to the ultrapower $B^\omega$ given by $\Phi\big((T_n)_n\big)= \big(\varphi_n(T_n)\big)_n$. Thus $\Phi\circ\Psi(a)= \big(\varphi_n\circ\psi_n(a)\big)_n= \big(\pi(a)\big)_n$, which we may identify with $\pi(a)$ itself by treating $B$ as a C*-subalgebra of $B^\omega$ through constant sequences. Thus $\Phi|_{\Psi(A)}$ must also be a \textasteriskcentered-homomorphism.
	
	Let $\epsilon>0$ and finite subset $\mathcal{G}\subset\Psi(A)$ of contractions 
	be given. Also let $\lambda:\mathrm{span}(\mathcal{G})\rightarrow A$ be a local lifting of $\Psi$. By nuclearity of $A\subseteq C$, there exist a finite-dimensional C*-algebra $G$ and c.c.p. maps $A\stackrel{\theta}{\longrightarrow}G\stackrel{\xi}{\longrightarrow}C$ such that, for every $d\in\mathcal{G}$,
	$$\|(\Psi\circ\xi)\circ(\theta\circ\lambda)(d)-d\|= \|\Psi\circ\xi\circ\theta\circ\lambda(d)-\Psi\circ\lambda(d)\|\leq \|\xi\circ\theta\circ\lambda(d)-\lambda(d)\|<\epsilon.$$
	Another use of Arveson's Extension Theorem yields $\theta':\Psi(A)\rightarrow G$ with restriction $\theta'|_\mathcal{G}= \theta\circ\lambda$. 
	
	Thus the inclusion $\Psi(A)\subseteq \prod_\omega F_n$ is nuclear. By the Choi-Effros Lifting Theorem \cite[Thm. C.3]{BrownOzawa}, said inclusion lifts to a c.c.p. map to $\prod_n F_n$, therefore $\Psi(A)$ is quasidiagonal (see eg. \cite[Exc. 7.1.3]{BrownOzawa}).
\end{proof}

\begin{remark}	
	It seems probable to the author that this theorem may be strengthened to show that $\pi$ admits an approximately multiplicative decomposition iff it factors through a quasidiagonal C*-algebra $D$ via $A\stackrel{\pi_1}{\rightarrow}D\stackrel{\pi_2}{\rightarrow}B$ such that $\pi_2$ admits an approximately multiplicative decomposition. All that is needed for this is to show that the trace $\tau\circ\pi_2$ is quasidiagonal for every trace $\tau$ on $\pi_2(D)= \pi(A)$. This is of course satisfied if every trace on $B$ is quasidiagonal, but the common theme of these results has been moving requirements away from the C*-algebras and onto the \textasteriskcentered-homomorphism itself.
\end{remark}

\section*{Acknowledgements}

The author would like to thank his advisor, Jos\'e Carri\'on, for much guidance 
and---alongside with Christopher Schafhauser---for asking the question that led to this paper.

\bibliographystyle{amsalpha}
\bibliography{DAW}

\end{document}